%% file: main.tex
\documentclass[reqno]{amsart}

\usepackage{amsrefs}

\usepackage{amsmath,amsfonts,amssymb,amsthm}
\numberwithin{equation}{section}
\usepackage{mathrsfs}
\usepackage[svgnames]{xcolor}
\usepackage[colorlinks,linkcolor=blue,citecolor=blue,urlcolor=Navy,breaklinks]{hyperref}

\usepackage{comment}
\usepackage{graphicx}
\usepackage{caption,subcaption}

\newtheorem{theorem}{Theorem}[section]
\newtheorem{proposition}[theorem]{Proposition}
\newtheorem{corollary}[theorem]{Corollary}
\newtheorem{lemma}[theorem]{Lemma}
\newtheorem{conjecture}[theorem]{Conjecture}
\newtheorem{problem}[theorem]{Problem}

\theoremstyle{definition}
\newtheorem{definition}[theorem]{Definition}
\newtheorem{example}[theorem]{Example}
\newtheorem{remark}[theorem]{Remark}

\title[Transversal Difference Numbers]{Transversal Difference Numbers in Finite Abelian Quotients}

\author[M. Barcau \and V. Pa\c sol]{Mugurel Barcau \and
Vicen\c tiu Pa\c sol}
\address{Institute of Mathematics of the Romanian Academy \and CertSIGN, Bucharest}
\email{mugurel.barcau@imar.ro; vicentiu.pasol@imar.ro}
\author[G. C. \c Turca\c s]{George C. \c Turca\c s}
\address{Babe\c s-Bolyai University, Cluj-Napoca \and certSIGN, Bucharest}
\email{george.turcas@ubbcluj.ro}

\thanks{This work was supported by the project “Group schemes, root systems, and
related representations” funded by the European Union - NextGenerationEU through
Romania’s National Recovery and Resilience Plan (PNRR) call no. PNRR-III-C9-2023-
I8, Project CF159/31.07.2023, and coordinated by the Ministry of Research, Innovation
and Digitalization (MCID) of Romania. All authors are also supported by certSign Research and Innovation.}

\begin{document}

\begin{abstract}
Given \(H\leq G\) finite abelian groups,  a transversal \(T\subseteq G\) for
\(G/H\) has fixed size \(|G/H|\), but its ambient difference support
\(D(T)=T-T\) can vary with the embedding of \(H\) in \(G\).  We call
$
\delta(G,H)=\min_T |D(T)|
$
the transversal difference number of the pair \((G,H)\).  This invariant is
related to finite abelian factorisation, tiling complements, and small-sumset
questions, and is motivated by recent work regarding ambient Galois labels in CRT transforms for
cyclotomic-subfield homomorphic encryption.  We prove various results regarding this invariant, including a general lower bound
$\delta(G,H)\geq 2|G/H|-m(G,H),
$
where \(m(G,H)\) is the largest order of a subgroup of \(G\) disjoint from
\(H\).  The bound is sharp for cyclic quotients, and Kneser's theorem gives a
cross-transversal estimate leading to exact product families with one nonsplit
cyclic coordinate and arbitrary split factors.  These results isolate the first genuinely new residual obstruction, namely the
same-prime square plane
\[
G=(\mathbb Z/p^2\mathbb Z)^2,\qquad H=pG.
\]
For odd \(p\), this case is the technical core of the paper. Here transversals are
graphs of functions \(\mathbb F_p^2\to \mathbb F_p^2\), and \(D(T)\) decomposes into
carry-corrected finite-field derivative images. We conjecture that
\[
\delta(G,H)=(2p-1)^2
\]
for all odd primes \(p\), prove the unconditional lower bound \(3p^2-p-1\), and give
small-prime, probabilistic, and fixed-polynomial evidence for the conjecture.
\end{abstract}
\maketitle 

\section{Introduction}

A quotient group $G/H$ remembers the cosets of \(H\), but it forgets how a choice of
coset representatives sits inside the ambient group.  This ambient information
can vary substantially from one section to another.  Throughout, all groups are
finite abelian and written additively.  Given
\(H\leq G\),  a transversal
\(T\subseteq G\) for \(G/H\) has cardinality \(|G/H|\). However, its \emph{difference support} $D(T)= T-T$ can varry substantially in size and a natural question is how small this set $D(T)$ can be. We call
\[
        \delta(G,H)
        =
        \min_T |D(T)|
\]
the \emph{transversal difference number} of the pair \((G,H)\), where \(T\)
ranges over all transversals for \(G/H\).

This invariant depends on the embedding \(H\leq G\), not only on the abstract
quotient.  For instance,
$\delta(C_2,0)=2$ and 
        $\delta(C_4,2C_4)=3$,
although both quotients are isomorphic to \(C_2\). When trying to find $\delta(G,H)$, transversals can easily be found, but proving that their pairwise differences are as small as possible in the ambient group can be challenging.

After translating \(T\) we may assume \(0\in T\); then \(T\) is a
set-theoretic tiling complement to \(H\), so every element of \(G\) is written
uniquely as \(h+t\) with \(h\in H\) and \(t\in T\).  Thus \(\delta(G,H)\) is a
support minimisation problem for finite abelian factorisations \(G=H\oplus T\),
where \(\oplus\) denotes unique representation rather than an internal direct
product.  This places it near the classical theory of unique-representation factorisations and algebraic tilings \cites{Hajos1942,Redei1965, Stein1974,Dinitz2006FullRankTilings, SzaboSands2009FactoringGroups}. The lower-bound methods are closest to small-sumset, critical-pair, and Kneser-type stabiliser arguments \cites{Kneser1955AbelscheGruppen, Kemperman1960, Freiman1973,TaoVu2006,Grynkiewicz2013StructuralAdditiveTheory,GreenRuzsa2007}.  It
is not a difference-set or a relative difference-set problem, where multiplicity conditions rather than support minimisation are central
\cites{Singer1938DifferenceSets,Butson1963RelativeDifferenceSets,PottSchmidtZhou2014}, nor a
difference-basis problem, where the set itself may vary in size while its
differences cover the ambient group
\cite{BanakhGavrylkiv2019DifferenceBases}.  Here \(T\) has fixed size and must
be a transversal; only the number of distinct ambient differences is minimised.

Our motivation comes from the Galois-theoretic bookkeeping that appears
in Peikert and Pepin's \emph{Vive Galois!} programme
\cite{PeikertPepinViveGalois1}.  In exact fully homomorphic encryption
one often wants many plaintext ``slots'' over a prescribed finite field
or finite ring, so that one ciphertext carries a vector of plaintexts
and homomorphic operations act componentwise on that vector.  This is the
SIMD idea of Smart and Vercauteren \cite{SmartVercauteren2014SIMD}.
Cyclotomic rings provide fast arithmetic, many automorphisms, and good
geometric properties, but their slot structure is arithmetically rigid:
the residue degree of the chosen plaintext prime determines the slot
type, and an unwanted extension degree gives fewer slots of a larger
field than the application naturally wants.  The point of the
Peikert--Pepin construction is that cyclotomic subfields, and more
generally abelian number rings, give a more flexible Galois-theoretic
way to obtain the desired slot type and number of slots, while still
supporting CRT bases, structured transforms, and packed bootstrapping.

A finite choice of representatives enters this construction quite
explicitly.  In a tower of abelian extensions \(M/L/K\), write
\[
        G_{M/K}=\operatorname{Gal}(M/K),
        \qquad
        G_{M/L}=\operatorname{Gal}(M/L).
\]
Peikert and Pepin choose a transversal
\(T\subseteq G_{M/K}\) for the quotient \(G_{M/K}/G_{M/L}\) in order to
index a CRT basis and the corresponding top-down sparse CRT transform
\cite{PeikertPepinViveGalois1}*{Def.~5.8}.
Their Lemma~5.13 \cite{PeikertPepinViveGalois1}*{Lemma~5.13} shows that, in the
relevant automorphism expansion, a coefficient can be nonzero only for an
ambient automorphism of the form
\[
        \tau=t't^{-1}
        \qquad (t,t'\in T).
\]
Thus the possible ambient automorphism labels are contained in \(TT^{-1}\).  In
the additive notation of this paper, the same support is \(T-T\).  Consequently,
the purely finite-abelian problem of
minimising \(|T-T|\) over quotient transversals is exactly the problem of
minimising the representative-dependent ambient label support forced by
this part of the Galois/CRT bookkeeping.

This is the precise sense in which a small value of \(\delta(G,H)\) can
be useful for the FHE framework.  In Peikert--Pepin's homomorphic CRT
transforms \cite{PeikertPepinViveGalois1}*{Secs.~3.2 and~4}, structured
linear maps are evaluated by applying automorphisms and taking linear
combinations.  Each distinct ambient automorphism label that is actually
used is a potential automorphism operation, and in an FHE implementation
such an operation typically comes with key-switching data and
key-switching cost.  The first and third phases of packed bootstrapping
are precisely CRT transforms: they move noisy decryption coefficients
into the SIMD slots, allow the slotwise nonlinear decoding step, and
then move the result back.  Hence, for an implementation model in which
each distinct ambient automorphism label carries cost, a transversal
with smaller \(TT^{-1}\) may reduce the candidate
automorphism/key-switch label set and simplify the linear part of packed
bootstrapping.

A word of caution about this motivation is worth mentioning.  The quantity \(\delta(G,H)\) measures
just the size of the support of ambient automorphism labels forced
by a choice of representatives, and therefore it should not be mistaken for a statement
about runtime, memory, key-switching cost, parameter selection, or security.
Several effects sit between this invariant and any real implementation.  The
label set \(TT^{-1}\) records which automorphisms \emph{can} occur, but some of
the corresponding coefficients may vanish for reasons the transversal cannot
see; the cost of applying an automorphism depends on the particular FHE
platform; and the Ring-LWE hardness assumptions used in this line of work are
unaffected by how small \(|T-T|\) is.  Once these caveats are set aside, what
remains is a clean, finite, and entirely algebraic question, and that is the one
we pursue: how small can the difference support of a section of a finite abelian
quotient be, what is its value across broad families, and where does the naive
split-or-cyclic intuition fail?

\medskip

\noindent\textbf{Main results.}
The first part of the paper develops the theory in the finite abelian setting.
Here the trivial lower bound \(\delta(G,H)\ge |G/H|\) is attained exactly when
\(H\) has a subgroup complement.  To go further, we isolate the singleton
quotient directions of a transversal: their canonical lifts form a subgroup of
\(G\) disjoint from \(H\), and quotienting by it makes the transversal
primitive, leaving no nonzero singleton fibres.  This reduction yields the
uniform lower bound
$ \delta(G,H)\ge 2|G/H|-m(G,H)$,
where \(m(G,H)\) is the largest order of a subgroup of \(G\) disjoint from
\(H\).  For cyclic quotients the bound is sharp, and feeding Kneser's theorem (Theorem \ref{thm:kneser})
into a cross-transversal estimate pins down exact product families with one
nonsplit cyclic coordinate and arbitrary split factors.  This cyclic sharpness
sits within the prime-cyclic critical-pair tradition initiated by Vosper
\cite{Vosper1956}, and more broadly within Kemperman's analysis of small
sumsets in abelian groups \cite{Kemperman1960}.  The precise statements are
given in Sections~\ref{sec:first-bounds-cyclic}
and~\ref{sec:chains-products-kneser}, and there results therein should be viewed as both a general theory and as delimitation
results: they identify the regimes in the ideas described above already determine the invariant. The main new difficulty begins
when these mechanisms no longer control the support, namely in same-prime residual
quotients with more than one nonsplit direction.

Section \ref{sec:noncyclic-obstructions} marks this transition. A residual quotient that contains a \(C_2\times C_2\)
plane already forces a strict gap from the trivial lower-bound. 

For odd primes, however, the parity argument disappears, and the first case left
open by all the preceding exact results is
\[
G_p=(\mathbb Z/p^2\mathbb Z)^2,\qquad H_p=pG_p.
\]
This odd square-plane problem, treated from Section \ref{sec:odd-square-plane} onward, is the technical core
of the paper.
Here transversals become graphs of functions \(\mathbb F_p^2\to\mathbb F_p^2\),
and the fibres of \(D(T)\) turn into carry-corrected finite-field derivative
images.  Read in these terms, the square-plane problem sits close to
finite-field sum-product \cite{BourgainKatzTao2004} and image-expansion
\cite{MurphyPetridis2017} questions, while the fixed-polynomial evidence
gathered in Section \ref{sec:asymptotic-evidence} lies nearest to value-set
questions for polynomial maps \cite{MullenWanWang2013}.  We conjecture that \(\delta(G,H)=(2p-1)^2\) and prove the unconditional bound \( \delta(G,H) \geq 3p^2-p-1\).  We also
prove that random liftings meet the lower bound $(2p-1)^2$ with probability tending to
one, and that no fixed integer-polynomial lifting can produce counterexamples
once the prime is large enough.  The final section gathers the
primitive-quotient and square-modulus open problems that emerge along the way.

\input{section2-preliminaries.tex}
\input{section3-with-primitive-quotient-reduction.tex}
\input{section4-chains-products-kneser-updated.tex}
\input{section5-noncyclic-obstructions.tex}
\input{section6-odd-square-plane.tex}

\input{section7-asymptotic-evidence.tex}
\input{section8-open-problems-and-appendices.tex}

\section*{Code and data availability}

The computations mentioned in
Subsection~\ref{subsec:small-prime-certificates} were carried using Python and are released in the form of an online companion on GitHub at the link below
\begin{center}
\url{https://github.com/georgeturcasubb/transversals} 
\end{center}

\section*{Use of Artificial Intelligence Tools}

The authors used OpenAI ChatGPT Deep Research, accessed through ChatGPT to help locate and organize potentially relevant bibliographic material. The authors also used OpenAI Codex with GPT-5.5, accessed through the Codex agentic programming environment, to assist with Python experiments, realization of the GitHub repository and copy-editing of the manuscript. These tools were used as research aids, not as authors. In line with the transparency and human-responsibility principles of the Leiden Declaration on Artificial Intelligence and Mathematics \cite{leidendeclaration2026}, the human authors checked the sources, code, numerical results and take full responsibility for the results presented in the article.

\bibliographystyle{alpha}
\bibliography{references}

\end{document}

%% file: section2-preliminaries.tex
\section{Preliminaries and notation}
\label{sec:preliminaries}

Throughout the paper all groups are finite abelian and are written additively.
For subsets $A,B\subseteq G$ we write
$
  A+B=\{a+b:a\in A,\ b\in B\}$,
$A-B=\{a-b:a\in A,\ b\in B\}$,
and $-A=\{-a:a\in A\}$.  If $K\leq G$ is a subgroup, then $A+K$ denotes
the union of the $K$-cosets meeting $A$.

Let $H\leq G$.  The quotient map will be denoted
\[
  \pi=\pi_{G,H}:G\longrightarrow G/H.
\]
When no confusion is possible we put
\[
  Q=G/H,\qquad q=|Q|=|G/H|.
\]
A \emph{transversal} for $G/H$ is a subset $T\subseteq G$ such that
$\pi|_T:T\to G/H$ is bijective.  Equivalently, $T$ contains exactly one element
from each coset of $H$ in $G$.

For a transversal $T$ we define its \emph{difference support}
\[
  D(T)=T-T.
\]
The invariant studied in this paper is the \emph{transversal difference number}
of the pair $(G,H)$:
\begin{equation} \label{eq:deltainv}
  \delta(G,H)=\min_T |D(T)|,
\end{equation}
where $T$ ranges over all transversals for $G/H$.  Translating a transversal by
an element of $G$ does not change its difference support.  Hence, whenever
convenient, we shall assume that $T$ is \emph{normalized}, meaning that
$0\in T$.

A normalized transversal is the same thing as a set-theoretic complement to
$H$: every element $g\in G$ has a unique expression
$g=h+t$,
  with $h\in H$, and $t\in T$.
This is a unique-representation factorisation of the set $G$; it does not mean
that $T$ is a subgroup unless this is explicitly stated. We use the terminology of set factorisations and complements in the
standard sense of \cite{SzaboSands2009FactoringGroups}; see also \cites{Hajos1942,Redei1965,Stein1974} for the classical factorisation and tiling background.

It will be useful to identify a transversal with its section of the quotient map.
Thus, for a normalized transversal $T$, we write
\begin{equation} \label{eq:sec}
  s=s_T:Q\longrightarrow G
\end{equation}
for the unique section such that $s(0)=0$, $\pi\circ s=\operatorname{id}_Q$, and
$T=s(Q)$.  For a quotient direction $a\in Q$ we define the corresponding
\emph{difference fibre}
\begin{equation} \label{eq:diff-fibre}
  \mathcal D_a(T)
  =D(T)\cap \pi^{-1}(a)
  =\{s(x+a)-s(x):x\in Q\}.
\end{equation}
The fibres $\mathcal D_a(T)$ are nonempty and decompose the difference support into the following disjoint union
$
  D(T)=\bigsqcup_{a\in Q}\mathcal D_a(T).
$
Note that in particular
$
  \mathcal D_0(T)=\{0\}$.
We call $a\in Q$ a \emph{singleton direction} for $T$ if
$|\mathcal D_a(T)|=1$.  The set of singleton directions will be denoted
\begin{equation} \label{eq:singleton-dir}
  \Sigma(T)=\{a\in Q:|\mathcal D_a(T)|=1\}.
\end{equation}
If $a\in \Sigma(T)$, the unique element of $\mathcal D_a(T)$ will sometimes be
written $d_T(a)$; equivalently,
$
  s(x+a)-s(x)=d_T(a)$ for all  $x\in Q$.
Therefore, $\pi(d_T(a))=a$.

We write
\[
  \widetilde\Sigma(T)=\{d_T(a):a\in\Sigma(T)\}\subseteq G
\]
for the corresponding lifted singleton set.  Thus $\Sigma(T)$ is a subset of
the quotient $G/H$, while $\widetilde\Sigma(T)$ is a subset of the ambient
group $G$.

\begin{definition}
\label{def:primitive-transversal}
A normalized transversal $T \subseteq G$ for $G/H$ is called \emph{primitive} if
  $\Sigma(T)=\{0\}$.
Equivalently, no nonzero quotient direction has a singleton difference fibre.
\end{definition}

The following subgroup invariant will also appear in our work:
\begin{equation} \label{eq:defmGH}
  m(G,H)=\max\{|K|:K\leq G,\ K\cap H=\{0\}\}.
\end{equation}
Note that $1\leq m(G,H)\leq |G/H|$ and the case in which the pair $G$, $H$ satisfies $m(G,H)=1$ will be called
\emph{residual}.  A subgroup $K\leq G$ is a subgroup complement to $H$ if
$K\cap H=\{0\}$ and $G=H+K$.  In particular, $H$ has a subgroup complement in $G$ if and only if
$m(G,H)=|G/H|$.

For later use we recall Kneser's theorem (see \cite{Kneser1955AbelscheGruppen} or \cite{TaoVu2006}*{Theorem 5.5}) in the finite form needed below.  If
$A\subseteq G$, its stabilizer is
\[
  \operatorname{Stab}(A)=\{g\in G:A+g=A\}.
\]

\begin{theorem}[Kneser]
\label{thm:kneser}
Let $A,B$ be nonempty subsets of a finite abelian group $G$, and write
  $P=\operatorname{Stab}(A+B)$.
Then
\[
  |A+B|\geq |A+P|+|B+P|-|P|.
\]
\end{theorem}

Since $A-B=A+(-B)$, the same estimate applies to difference sets
$A-B$ with $P=\operatorname{Stab}(A-B)$. The reader interested in broader background on Kneser-type inequalities, critical
pairs, and small-doubling methods in abelian groups, could consult
\cite{Kemperman1960}, \cite {TaoVu2006} or \cite{Grynkiewicz2013StructuralAdditiveTheory}.

%% file: section3-with-primitive-quotient-reduction.tex
\section{First bounds, singleton reduction and cyclic quotients}
\label{sec:first-bounds-cyclic}

The first estimates come from decomposing the difference support over quotient
directions. We notice that singleton fibres lift to a subgroup of
$G$ disjoint from $H$ and  quotienting by this subgroup makes the transversal
\textit{primitive}, gives the uniform lower bound in terms of $m(G,H)$, and leads to an
exact formula for cyclic quotients.

The first proposition below is the support-minimisation analogue of the basic
exact-factorisation question: when does a set-theoretic complement become
a subgroup complement? Classical factorisation results of Hajós \cite{Hajos1942} and Rédei \cite{Redei1965} are related to this phenomenon.

\begin{proposition}
\label{prop:complement-criterion}
Let $H\leq G$.  Then
$
  \delta(G,H)=|G/H|$
if and only if $H$ has a subgroup complement in $G$.
\end{proposition}

\begin{proof}
Let $q=|G/H|$.  For every transversal $T$ and every fixed $t_0\in T$, the
translate $T-t_0$ has $q$ elements and is contained in $D(T)=T-T$.  Hence
$|D(T)|\ge q$ for every $T$, so $\delta(G,H)\ge q$.

If $K\leq G$ is a subgroup complement to $H$, then $K$ is a transversal for
$G/H$ and $D(K)=K-K=K$. Thus $\delta(G,H)\le |K|=q$, and equality follows.

Conversely, suppose $\delta(G,H)=q$, and choose a transversal $T$ with
$|D(T)|=q$.  Translating $T$ if necessary, assume $0\in T$.  Then
$T\subseteq D(T)$ because $t=t-0$ for every $t\in T$.  Since the two sets have
the same cardinality, $T=D(T)$.  Therefore $T$ is closed under subtraction.
As $T$ is finite and contains $0$, it is a subgroup of $G$.  Since $T$ is also a
transversal for $G/H$, it is a subgroup complement to $H$.
\end{proof}

\subsection{Singleton directions and primitive quotient transversals}
\label{subsec:singleton-primitive-quotient}

We use the notation $\Sigma(T)$, $d_T(a)$, and $\widetilde\Sigma(T)$ from
Section~\ref{sec:preliminaries}.  The next proposition explains the quotient
reduction behind primitive transversals.

\begin{proposition}
\label{prop:primitive-quotient-reduction}
Let $H\leq G$, and let $T$ be a normalized transversal for $G/H$.  Then:
\begin{enumerate}
  \item $\Sigma(T)$ is a subgroup of $G/H$;

  \item the map
  \[
    \Sigma(T)\longrightarrow G,
    \qquad a\longmapsto d_T(a),
  \]
  is an injective group homomorphism, its image $\widetilde\Sigma(T)$ is a
  subgroup of $G$, and
  \[
    \widetilde\Sigma(T)\cap H=\{0\},
    \qquad
    |\widetilde\Sigma(T)|=|\Sigma(T)|;
  \]

  \item one has
  \[
    T+\widetilde\Sigma(T)=T,
    \qquad
    D(T)+\widetilde\Sigma(T)=D(T),
  \]
  so $D(T)$ is a union of cosets modulo $\widetilde\Sigma(T)$;

  \item if
  \[
    \overline G=G/\widetilde\Sigma(T),
    \qquad
    \overline H=(H+\widetilde\Sigma(T))/\widetilde\Sigma(T),
  \]
  and $\overline T$ denotes the image of $T$ in $\overline G$, then
  $\overline T$ is a transversal for $\overline G/\overline H$ and
  \[
    |D(T)|=|\widetilde\Sigma(T)|\,|D(\overline T)|;
  \]

  \item the transversal $\overline T$ has no nonzero singleton directions.
\end{enumerate}
\end{proposition}

\begin{proof}
Put $Q=G/H$ and write $s=s_T$ for the normalized section attached to $T$.
The zero direction is singleton, since $\mathcal D_0(T)=\{0\}$.  Let
$a,b\in\Sigma(T)$.  Then, for every $x\in Q$,
\[
  s(x+a)-s(x)=d_T(a),
  \qquad
  s(x+b)-s(x)=d_T(b).
\]
Hence
\[
\begin{aligned}
  s(x+a+b)-s(x)
  &=\bigl(s(x+a+b)-s(x+a)\bigr)+\bigl(s(x+a)-s(x)\bigr) \\
  &=d_T(b)+d_T(a).
\end{aligned}
\]
Thus $a+b$ is a singleton direction and
\[
  d_T(a+b)=d_T(a)+d_T(b).
\]
Since $Q$ is finite, closure under addition and the presence of $0$ imply that
$\Sigma(T)$ is a subgroup.  The displayed identity also shows that
$a\mapsto d_T(a)$ is a group homomorphism.  This proves (1) and the homomorphism
assertion in (2).

For $a\in\Sigma(T)$ we have
\[
  \pi(d_T(a))=a,
\]
because $d_T(a)=s(x+a)-s(x)$ for every $x\in Q$.  Hence the homomorphism
$a\mapsto d_T(a)$ is injective, and its image $\widetilde\Sigma(T)$ is a
subgroup of $G$ with $|\widetilde\Sigma(T)|=|\Sigma(T)|$.  If
$d_T(a)\in H$, then $a=\pi(d_T(a))=0$, hence $d_T(a)=0$.  Therefore
\[
  \widetilde\Sigma(T)\cap H=\{0\}.
\]
This proves (2).

If $\sigma=d_T(a)\in\widetilde\Sigma(T)$, then
\[
  s(x+a)=s(x)+\sigma
\]
for every $x\in Q$.  As $x$ varies over $Q$, so does $x+a$, and therefore
translation by $\sigma$ permutes the elements of $T$.  Thus
\[
  T+\widetilde\Sigma(T)=T.
\]
Subtracting gives
\[
  D(T)+\widetilde\Sigma(T)=D(T),
\]
so $D(T)$ is a union of cosets modulo $\widetilde\Sigma(T)$.

Let $q_{\natural}:G\to\overline G$ be the quotient map.  We prove (4).  The
image $\overline T=q_{\natural}(T)$ plainly meets every
$\overline H$-coset.  For uniqueness, suppose that $q_{\natural}(t_1)$ and
$q_{\natural}(t_2)$ lie in the same $\overline H$-coset.  Then
\[
  t_1-t_2\in H+\widetilde\Sigma(T).
\]
Write $t_1-t_2=h+\sigma$, with $h\in H$ and
$\sigma\in\widetilde\Sigma(T)$.  Since $T+\widetilde\Sigma(T)=T$, the element
$t_2+\sigma$ belongs to $T$.  But
\[
  t_1-(t_2+\sigma)=h\in H.
\]
Since $T$ is a transversal for $G/H$, this forces $t_1=t_2+\sigma$, hence
$q_{\natural}(t_1)=q_{\natural}(t_2)$.  Thus $\overline T$ is a transversal.

Moreover,
\[
  q_{\natural}(D(T))=q_{\natural}(T-T)=\overline T-\overline T=D(\overline T).
\]
Since $D(T)$ is a union of $\widetilde\Sigma(T)$-cosets, each element of
$D(\overline T)$ has exactly $|\widetilde\Sigma(T)|$ preimages in $D(T)$.  Hence
\[
  |D(T)|=|\widetilde\Sigma(T)|\,|D(\overline T)|.
\]

It remains to show that $\overline T$ is primitive.  Under the natural
identification
\[
  \overline G/\overline H
  \simeq
  G/(H+\widetilde\Sigma(T))
  \simeq
  (G/H)/\Sigma(T),
\]
suppose that a nonzero class $\overline a$ has singleton fibre for
$\overline T$, and choose a representative $a\in Q$.  Then, for some $d\in G$,
\[
  \mathcal D_a(T)\subseteq d+\widetilde\Sigma(T).
\]
Indeed, the quotient fibre in direction \(\overline a\) is the image of the
fibres \(\mathcal D_{a+k}(T)\), \(k\in\Sigma(T)\); if that image is a singleton,
then each such fibre lies in one \(\widetilde\Sigma(T)\)-coset.
For every $k\in\Sigma(T)$, the fibre in direction $a+k$ satisfies
\[
\begin{aligned}
  \mathcal D_{a+k}(T)
  &=\{s(x+a+k)-s(x):x\in Q\} \\
  &=\{s(x+a)+d_T(k)-s(x):x\in Q\} \\
  &=\mathcal D_a(T)+d_T(k).
\end{aligned}
\]
Thus the union of the fibres above $a+\Sigma(T)$ lies in the single coset
$d+\widetilde\Sigma(T)$, whose size is
$|\widetilde\Sigma(T)|=|\Sigma(T)|$.  The $|\Sigma(T)|$ fibres in this union are
nonempty and disjoint because they lie over distinct quotient directions.
Hence each is a singleton.  In particular $a\in\Sigma(T)$, so
$\overline a=0$, a contradiction.  Thus $\overline T$ has no nonzero singleton
directions.
\end{proof}

Every transversal becomes primitive after quotienting by its lifted singleton
subgroup.  This is the structural reason for the lower bound below.

\begin{proposition}
\label{prop:quotient-fibre-lower-bound}
For every $H\leq G$ one has
\[
  \delta(G,H)\geq 2|G/H|-m(G,H).
\]
More precisely, for every normalized transversal $T$ one has
\[
  |D(T)|\geq 2|G/H|-|\Sigma(T)|.
\]
\end{proposition}

\begin{proof}
Let $T$ be a normalized transversal for $G/H$.  Put $Q=G/H$, $q=|Q|$, and
$k=|\Sigma(T)|$.  By the fibre decomposition,
\[
  |D(T)|
  =\sum_{a\in Q}|\mathcal D_a(T)|
  \geq k+2(q-k)
  =2q-k,
\]
because the $k$ singleton fibres have size $1$ and all other fibres are nonempty
of size at least $2$.  This proves the more precise bound.  By
Proposition~\ref{prop:primitive-quotient-reduction}, the lifted singleton
subgroup $\widetilde\Sigma(T)$ is disjoint from $H$, so
$k=|\widetilde\Sigma(T)|\leq m(G,H)$.
Thus
\[
  |D(T)|\geq 2q-m(G,H).
\]
Taking the minimum over all transversals proves the proposition.
\end{proof}

\begin{remark}
Looking at the proof above, if $T$ is a normalized
transversal, it is easy to see that the equality
\[
  |D(T)|=2|G/H|-m(G,H)
\]
holds if and only if
\[
  |\Sigma(T)|=m(G,H)
\]
and every non-singleton quotient-direction fibre $\mathcal D_a(T)$ has size
exactly $2$.
\end{remark}

\subsection{Cyclic quotients}
\label{subsec:cyclic-quotients}

The bound of Proposition~\ref{prop:quotient-fibre-lower-bound} is sharp for all
cyclic quotients.
\begin{theorem}
\label{thm:cyclic-quotient-formula}
Let $H\leq G$, and suppose that $G/H$ is cyclic of order $q$.  Then
\[
  \delta(G,H)=2q-m(G,H).
\]
\end{theorem}

\begin{proof}
The lower bound is Proposition~\ref{prop:quotient-fibre-lower-bound}.
We first handle the residual case \(m(G,H)=1\).  Choose \(g\in G\) whose image
generates \(G/H\).  If \(q=1\), there is nothing to prove.  For \(q>1\), the
integer \(q\) divides \(\operatorname{ord}(g)\).  The equality
\(\operatorname{ord}(g)=q\) would make \(\langle g\rangle\) a nontrivial
subgroup disjoint from \(H\), contradicting residuality.  Hence
$
  \operatorname{ord}(g)\geq 2q.
$

Now observe that
$
  T=\{0,g,2g,\ldots,(q-1)g\}
$
is a transversal for \(G/H\), and
$
  D(T)=\{kg:-(q-1)\leq k\leq q-1\}
$.
These \(2q-1\) elements are distinct because any two exponents in the displayed
range differ by an integer of absolute value at most \(2q-2\), which is
strictly smaller than \(\operatorname{ord}(g)\).  Thus
\(\delta(G,H)=2q-1\) in the residual cyclic case.

For the general case, let
$K\leq G$ be a subgroup disjoint from $H$ with $|K|=m(G,H)$, and write
$m=|K|$.  Since $G/H$ is cyclic, the image of $K$ in $G/H$ is the unique
subgroup of order $m$.  Hence
\[
  (G/K)/((H+K)/K)\simeq G/(H+K)
\]
is cyclic of order $s=q/m$.

The pair $(G/K,\,(H+K)/K)$ is residual.  Indeed, if $L/K\leq G/K$ were nontrivial and disjoint from
$(H+K)/K$, then $L\cap(H+K)=K$.
In particular, $L\cap H=\{0\}$, because $K\cap H=\{0\}$.  But then $L$ would be
a subgroup of $G$ disjoint from $H$ and strictly larger than $K$, contradicting
the maximality of $K$. 

By the residual case just proved, there is a transversal
$\overline B\subseteq G/K$ for
\[
  (G/K)/((H+K)/K)
\]
such that
\[
  |\overline B-\overline B|=2s-1.
\]
Choose representatives $B\subseteq G$ for $\overline B$, and define
\[
  T=B+K.
\]
Since $B$ represents the quotient $G/(H+K)$ and $K$ maps isomorphically onto its
image in $G/H$, the set $T$ is a transversal for $G/H$.  Moreover,
\[
  D(T)=(B-B)+K.
\]
Modulo $K$, this support is exactly $\overline B-\overline B$, so it is a union
of $2s-1$ distinct $K$-cosets.  Therefore
\[
  |D(T)|=|K|(2s-1)=m\left(2\frac qm-1\right)=2q-m.
\]
Together with the lower bound, this proves the formula.
\end{proof}

For a cyclic ambient group with a prescribed subgroup, the formula becomes
completely explicit. In what follows, we write $C_n$ for the cyclic subgroup of order $n$, for any $n \geq 1$.

\begin{corollary}
\label{cor:cyclic-fixed-stage-formula}
Let $h,q\geq 1$, let $G=C_{hq}$, and let $H\leq G$ be the subgroup of order
$h$.  If $c$ is the largest divisor of $q$ that is coprime to $h$, then
\[
  \delta(G,H)=2q-c.
\]
\end{corollary}

\begin{proof}
In a cyclic group there is a unique subgroup of each order.  A subgroup of
order $k$ intersects $H$ trivially if and only if $\gcd(k,h)=1$.  Since such a
subgroup must inject into $G/H$, its order must also divide $q$.  Therefore
$m(G,H)$ is exactly the largest divisor $c$ of $q$ that is coprime to $h$.  The
claim follows from Theorem~\ref{thm:cyclic-quotient-formula}.
\end{proof}

\begin{example}[Mixed-prime square-modulus products]
\label{ex:mixed-prime-square-modulus-formula}
Let $p_1,\ldots,p_s$ be distinct primes, let $n=\prod_i p_i$, and put
\[
  G=\prod_{i=1}^s C_{p_i^2},
  \qquad
  H=\prod_{i=1}^s p_iC_{p_i^2}.
\]
Then
\[
  \delta(G,H)=2n-1.
\]
Indeed, the Chinese remainder theorem identifies $(G,H)$ with
$
  (C_{n^2},\,nC_{n^2}).
$
Here $|H|=n$ and $|G/H|=n$.  In
Corollary~\ref{cor:cyclic-fixed-stage-formula}, with $h=q=n$, the largest
divisor of $q$ coprime to $h$ is $1$.  Hence
$
  \delta(G,H)=2n-1
$.
Thus square-modulus products with distinct prime factors behave cyclically.  The
first obstruction not covered by the cyclic analysis is therefore not a product
over distinct primes, but a same-prime \(p\)-primary quotient of rank at least
two.  This is the source of the square-plane problem studied later in the manuscript.
\end{example}

Such CRT and product constructions are frequent in the algebraic-tiling
and finite-abelian factorisation literature
\cites{Stein1974,Dinitz2006FullRankTilings,SzaboSands2009FactoringGroups}.

%% file: section4-chains-products-kneser-updated.tex
\section{Chains, products, and cross-transversal estimates}
\label{sec:chains-products-kneser}

The aim of this section is to study the behaviour of the invariant $\delta$ with respect to subgroup chains and direct products.  The technique we use for this study involves
slicing a transversal over an intermediate quotient.  Kneser's theorem then
upgrades the elementary nonzero-fibre count and gives exact product families
with one nonsplit cyclic coordinate and arbitrary split factors. The resulting estimates are cross-transversal bounds based on Kneser's theorem, and
should be read against the structural additive theory of small sumsets in
finite abelian groups \cites{Kemperman1960,TaoVu2006,Grynkiewicz2013StructuralAdditiveTheory}.

We start with some elementary chain bounds: a multiplicative construction and
two lower bounds coming from the zero and nonzero quotient fibres.
\begin{proposition}
\label{prop:chain-inequalities}
Let
$
  H\leq K\leq G$,
and write
\[
  d_0=\delta(K,H),
  \qquad
  d_1=\delta(G,K),
  \qquad
  r=|K/H|,
  \qquad
  q=|G/K|.
\]
Then
\[
\max\{d_0+d_1-1,\ d_0+(q-1)r\} \leq  \delta(G,H)\leq d_0d_1.
\]
\end{proposition}

\begin{proof}
Choose a transversal $V\subseteq K$ for $K/H$ with $|D(V)|=d_0$, and choose a
transversal $U\subseteq G$ for $G/K$ with $|D(U)|=d_1$.  Then
\[
  T=U+V
\]
is a transversal for $G/H$.  Moreover
\[
  D(T)=(U+V)-(U+V)\subseteq D(U)+D(V).
\]
Thus
\[
  \delta(G,H)\leq |D(T)|\leq |D(U)|\,|D(V)|=d_1d_0.
\]

For the first lower bound, let $T$ be a transversal for $G/H$, and write
\(T_y=T\cap \pi_{G,K}^{-1}(y)\) for its slice over \(y\in G/K\).  Choose one
element of each nonempty slice; these elements form a transversal $U_T$ for
$G/K$, so $|D(U_T)|\ge d_1$.  Each translated slice is a transversal for
\(K/H\), so the zero $G/K$-fibre of $D(T)$ contains at least \(d_0\) elements.
Since $D(U_T)$ meets that zero fibre only in $0$,
\[
  |D(T)|\geq d_0+d_1-1.
\]

For the second lower bound, one translated slice gives at least $d_0$ elements
in the zero fibre.  In every nonzero $G/K$-direction, the cross-difference
between two corresponding slices maps onto all of \(K/H\): the two slices each
meet every \(H\)-coset inside their \(K\)-coset.  Hence that fibre has at least
$r=|K/H|$ elements, and
\[
  |D(T)|\geq d_0+(q-1)r.
\]
Taking the minimum over all $T$ proves the proposition.
\end{proof}

The proof above isolates the point where information is lost: in each nonzero
$G/K$-direction it uses only that a cross-difference of two slices surjects onto
$K/H$.  Kneser's theorem gives a uniform replacement for this crude fibre count.

\begin{theorem}
\label{thm:kneser-cross-transversal}
Let $H\leq G$, put $q=|G/H|$, and let $A,B\subseteq G$ be transversals for
$G/H$.  Then
\[
  |A-B|\geq 2|G/H|-m(G,H).
\]
\end{theorem}

\begin{proof}
Let $P=\operatorname{Stab}(A-B)$, and let $\pi:G\to G/H$ be the quotient map.
Denote by
$s=|P\cap H|$,
and by
  $u=|\pi(P)|$.
The exact sequence
\[
  0\longrightarrow P\cap H\longrightarrow P\longrightarrow \pi(P) \longrightarrow 0
\]
gives
  $|P|=su$.
By Kneser's theorem, applied to $A+(-B)$,
\[
  |A-B|\geq |A+P|+|B+P|-|P|.
\]
Since $A$ contains one element in each $H$-coset, the set $A+P$ contains, in
each $H$-coset, at least one coset of $P\cap H$.  Hence
\[
  |A+P|\geq sq.
\]
The same argument gives $|B+P|\geq sq$ and therefore
  $|A-B|\geq s(2q-u)$.

If $s=1$, then $P\cap H=\{0\}$, so $P$ is a subgroup of $G$ disjoint from $H$.
Thus
\[
  u=|P|\leq m(G,H),
\]
and hence
\[
  |A-B|\geq 2q-u\geq 2q-m(G,H).
\]
If $s>1$, then $u\leq q$, and so
\[
  |A-B|\geq s(2q-u)\geq sq\geq 2q.
\]
Since $m(G,H)\geq 1$, this implies
\[
  |A-B|\geq 2q\geq 2q-m(G,H)
\]
and the theorem follows.
\end{proof}

Applying the theorem to pairs of slices in the chain $H\leq K\leq G$ replaces
the elementary contribution $|K/H|$ in every nonzero $G/K$-fibre by
$2|K/H|-m(K,H)$.
\begin{corollary}
\label{cor:kneser-sharpened-chain-bound}
Let $H\leq K\leq G$, and put
\[
  \begin{aligned}
    d_0&=\delta(K,H),
    & d_1&=\delta(G,K),
    r&=|K/H|,
    & q&=|G/K|,
    & m_0&=m(K,H).
  \end{aligned}
\]
Then
\[
  \delta(G,H)
  \geq
  \max\{d_0+d_1-1,\ d_0+(q-1)(2r-m_0)\}.
\]
\end{corollary}

\begin{proof}
The bound $d_0+d_1-1$ is Proposition~\ref{prop:chain-inequalities}.  Let $T$ be
a transversal for $G/H$ and slice it above the cosets of $K$.  The zero
$G/K$-fibre contains the difference set of one translated slice and contributes
at least $d_0$ elements.

For a nonzero direction in $G/K$, the corresponding fibre contains a translated
cross-difference between two slices.  After translation into $K$, these are
transversals for $K/H$, so Theorem~\ref{thm:kneser-cross-transversal} gives at
least
\[
  2r-m_0
\]
elements.  Summing over the $q-1$ nonzero directions in $G/K$ gives
\[
  |D(T)|\geq d_0+(q-1)(2r-m_0).
\]
Taking the minimum over $T$ proves the claim.
\end{proof}

\begin{remark}
Passing to quotients contained in $H$ can only identify differences.  If
$N\leq H\leq G$, then
\[
  \delta(G/N,H/N)\leq \delta(G,H).
\]
Indeed, the image of a transversal for \(G/H\) under the quotient map
\(G\to G/N\) is a transversal for \((G/N)/(H/N)\), and quotienting can only
identify differences.

Proposition~\ref{prop:primitive-quotient-reduction} gives a more precise form of
quotienting for the particular subgroup $\widetilde\Sigma(T)$ attached to a
transversal $T$: in that case the image transversal is primitive and the
difference support satisfies the exact multiplicative identity
\[
  |D(T)|=|\widetilde\Sigma(T)|\,|D(\overline T)|.
\]
\end{remark}

Applying the chain estimates to the two natural coordinate chains gives the
following elementary product bounds.
Let $H_i\leq G_i$, put
\[
  q_i=|G_i/H_i|,
  \qquad d_i=\delta(G_i,H_i)
  \qquad (i=1,2).
\]
Then
\begin{equation} \label{eq:bound-dir-products}
\max\{d_1+d_2-1,\ d_1+(q_2-1)q_1,d_2+(q_1-1)q_2\}\leq  \delta(G_1\times G_2,H_1\times H_2)\leq d_1d_2.
\end{equation}
The product of optimal transversals gives the upper bound: if
$T_i\subseteq G_i$ is a transversal for $G_i/H_i$ with $|D(T_i)|=d_i$, then
$T_1\times T_2$ is a transversal for
\[
  (G_1\times G_2)/(H_1\times H_2),
\]
and
\[
  D(T_1\times T_2)=D(T_1)\times D(T_2).
\]
Thus the difference support has size $d_1d_2$.

For the lower bounds, apply Proposition~\ref{prop:chain-inequalities} to the
chain
\[
  H_1\times H_2\leq G_1\times H_2\leq G_1\times G_2
\]
and then to the symmetric chain
\[
  H_1\times H_2\leq H_1\times G_2\leq G_1\times G_2.
\]
The identities
\[
  \delta(G_1\times H_2,H_1\times H_2)=\delta(G_1,H_1),
  \qquad
  \delta(G_1\times G_2,G_1\times H_2)=\delta(G_2,H_2)
\]
follow from projection and the fixed-coordinate construction.  Substituting
them in the two chains gives the displayed estimates.

The same sharpening applies to products through the two coordinate chains.
With \(m_i=m(G_i,H_i)\), Corollary~\ref{cor:kneser-sharpened-chain-bound}
gives
\[
\begin{split}
  \delta(G_1\times G_2,H_1\times H_2)
  \geq \max\{&d_1+d_2-1,\ d_1+(q_2-1)(2q_1-m_1),\\
              &d_2+(q_1-1)(2q_2-m_2)\}.
\end{split}
\]

\begin{remark}
Since $m_i\leq q_i$, the Kneser-sharpened product terms dominate the
corresponding elementary product terms above.  The improvement is strict in the term where the factor providing
$m_i$ is nonsplit, meaning $m_i<q_i$, and the other quotient is nontrivial.
\end{remark}

The Cartesian-product construction need not be optimal.
\begin{example}[Products need not be multiplicative]
\label{ex:products-not-multiplicative}
The cyclic formula gives
\[
  \delta(C_4,2C_4)=3,
  \qquad
  \delta(C_9,3C_9)=5.
\]
However
\[
  (C_4\times C_9,\ 2C_4\times 3C_9)
  \simeq
  (C_{36},\ 6C_{36}).
\]
By Corollary~\ref{cor:cyclic-fixed-stage-formula},
\[
  \delta(C_{36},6C_{36})=11.
\]
Thus the optimal value for the product pair is $11$, not $15$.
\end{example}

The cross-transversal estimate is exact after adjoining a split factor, provided
the base pair already attains the lower bound.
\begin{corollary}
\label{cor:adjoining-split-factor}
Let $H_0\leq G_0$, let $L$ be a finite abelian group, for which we write
$q_0=|G_0/H_0|$,
and
  $m_0=m(G_0,H_0)$.
Then
\[
  \delta(G_0\times L,H_0\times\{0\})
  \geq
  |L|(2q_0-m_0).
\]
If $\delta(G_0,H_0)=2q_0-m_0$,
the inequality above holds with equality.
\end{corollary}

\begin{proof}
Let $T$ be a transversal for
\[
  (G_0\times L)/(H_0\times\{0\}).
\]
For each $\ell\in L$, define
\[
  A_\ell=\{g\in G_0:(g,\ell)\in T\}.
\]
Because quotient cosets are indexed by $(G_0/H_0)\times L$, each $A_\ell$ is a
transversal for $G_0/H_0$.

For each $\ell\in L$, the $L$-coordinate $\ell$ fibre of $D(T)$ contains
\[
  (A_\ell-A_0)\times\{\ell\}.
\]
By Theorem~\ref{thm:kneser-cross-transversal},
\[
  |A_\ell-A_0|\geq 2q_0-m_0.
\]
These fibres are disjoint for different $\ell$, so
\[
  |D(T)|\geq |L|(2q_0-m_0).
\]
Taking the minimum over $T$ proves the lower bound.

For the upper bound under the displayed hypothesis, choose a transversal
$T_0\subseteq G_0$ for $G_0/H_0$ with
\[
  |D(T_0)|=2q_0-m_0.
\]
Then $T_0\times L$ is a transversal for
$(G_0\times L)/(H_0\times\{0\})$, and
\[
  D(T_0\times L)=D(T_0)\times L.
\]
Thus
\[
  |D(T_0\times L)|=|L|\,|D(T_0)|=|L|(2q_0-m_0),
\]
which proves equality.
\end{proof}

Together with the cyclic quotient formula, this gives the main exact product
family of the section.
\begin{corollary}
\label{cor:one-nonsplit-cyclic-coordinate}
Let $p$ be prime, let $1\leq b<a$, and let $L$ be a finite abelian group.  Then
\[
  \delta(C_{p^a}\times L,\ p^b C_{p^a}\times\{0\})
  =
  |L|(2p^b-1).
\]
\end{corollary}

\begin{proof}
For the base pair
\[
  G_0=C_{p^a},
  \qquad
  H_0=p^bC_{p^a},
\]
the quotient has size $p^b$.  Since $1\leq b<a$, the subgroup $H_0$ is nonzero,
and every nontrivial subgroup of the cyclic $p$-group $C_{p^a}$ meets $H_0$
nontrivially.  Therefore
\[
  m(G_0,H_0)=1.
\]
The cyclic quotient formula, Theorem~\ref{thm:cyclic-quotient-formula}, gives
\[
  \delta(G_0,H_0)=2p^b-1=2|G_0/H_0|-m(G_0,H_0).
\]
Corollary~\ref{cor:adjoining-split-factor} now gives the desired equality after
adjoining the split factor $L$.
\end{proof}

\begin{remark}
Corollary~\ref{cor:one-nonsplit-cyclic-coordinate} is not a product
multiplicativity statement: the added factor is split, with subgroup
$p^bC_{p^a}\times\{0\}$.  It does not cover a second nonsplit $p$-primary
coordinate, which is why
\(((\mathbb Z/p^2\mathbb Z)^2,\ p(\mathbb Z/p^2\mathbb Z)^2)\) is the first
genuinely new same-prime case.
\end{remark}

%% file: section5-noncyclic-obstructions.tex
\section{Noncyclic residual obstructions}
\label{sec:noncyclic-obstructions}

The cyclic quotient formula shows that the fibre lower bound is sharp for every
cyclic quotient.  Noncyclic residual quotients behave differently: the first
obstruction is already visible in a quotient plane of order four, and it
separates the cyclic theory from the higher-rank phenomena studied later.

The two-torsion obstruction gives a strict improvement whenever the quotient
contains a two-dimensional \(2\)-torsion plane.
\begin{theorem}
\label{thm:residual-two-torsion-obstruction}
Let $H\leq G$ be residual, so that $m(G,H)=1$.  Suppose that $G/H$ contains a
subgroup isomorphic to $C_2\times C_2$.  Then every transversal $T$ for $G/H$
satisfies
\[
  |D(T)|\geq 2|G/H|+1.
\]
In particular, the residual fibre lower bound $2|G/H|-1$ is not sharp for such a
pair.
\end{theorem}

\begin{proof}
Let $Q=G/H$, let $q=|Q|$, and normalize $T$, with associated section
$s:Q\to G$.  Since the pair is residual, Proposition~\ref{prop:primitive-quotient-reduction}
implies that no nonzero quotient direction can be singleton.  Therefore
\[
  |\mathcal D_a(T)|\geq 2
  \qquad(a\in Q,\ a\neq 0).
\]

For directions of order two there is a stronger parity fact.  If $a\in Q$ has
order two, then $a=-a$, hence
\[
  \mathcal D_a(T)=-\mathcal D_a(T).
\]
No element of this fibre can be fixed by negation: if $d=-d$, then $2d=0$ and
$\pi(d)=a\neq 0$, so $\langle d\rangle$ is a nontrivial subgroup disjoint from
$H$.  Thus every order-two direction has an even fibre of size at least $2$.

Choose a subgroup
\[
  P=\{0,a,b,c\}\leq Q,
  \qquad c=a+b,
\]
with $P\simeq C_2\times C_2$.  We claim that the three nonzero fibres over
$a,b,c$ cannot all have size $2$.  Suppose, to the contrary, that they do.  Put
\[
  A=s(a),\qquad B=s(b),\qquad C=s(c).
\]
Since each of the three fibres is stable under negation and has two elements,
we have
\[
  C-B=\varepsilon A,
  \qquad
  C-A=\eta B,
  \qquad
  B-A=\theta C
\]
for some signs $\varepsilon,\eta,\theta\in\{\pm1\}$.  The first two identities give
\[
  B+\varepsilon A=A+\eta B.
\]
If $\varepsilon=1$ and $\eta=-1$, this gives $2B=0$.  If $\varepsilon=-1$ and
$\eta=1$, it gives $2A=0$.  If $\varepsilon=-1$ and $\eta=-1$, it gives
$2(A-B)=0$.  Finally, if $\varepsilon=\eta=1$, then $C=A+B$; using
$B-A=\theta C$ gives either $2A=0$ or $2B=0$.  In all cases, one of $A$, $B$,
or $A-B$ has order two.  Its image in $Q$ is respectively $a$, $b$, or $a+b$,
hence nonzero.  The subgroup it generates is therefore nontrivial and disjoint
from $H$, again contradicting residuality.  Thus at least one of the three
fibres over $a,b,c$ has size at least $4$.

Counting quotient-direction fibres now gives
\[
\begin{aligned}
  |D(T)|
  &=\sum_{u\in Q}|\mathcal D_u(T)| \\
  &\geq 1+\bigl(2+2+4\bigr)+2(q-4) \\
  &=2q+1.
\end{aligned}
\]
This proves the theorem.
\end{proof}

For square-modulus \(2\)-groups this gives an exact value in rank two. The follwing corollary is immediate.
\begin{corollary}
\label{cor:two-torsion-square-modulus}
Let \(r\geq 2\), let
\[
  G=(\mathbb Z/4\mathbb Z)^r,
  \qquad H=2G.
\]
Then
\[
  \delta(G,H)\geq 2^{r+1}+1.
\]
For \(r=2\), the equality $\delta(G,H)=9$ is realized by the transversal $T=\{0,1 \}^2$.
\end{corollary}

The preceding obstruction is specific to quotient directions of order two.  For
odd primes there is no analogous negation parity argument.  The same
square-modulus construction is residual for every prime, and the coordinate box
gives the basic upper bound.
\begin{proposition}
\label{prop:residual-square-modulus-pairs}
Let $p$ be a prime and let
\[
  G=(\mathbb Z/p^2\mathbb Z)^r,
  \qquad
  H=pG.
\]
Then $m(G,H)=1$.  Moreover the coordinate box gives
\[
  \delta(G,H)\leq (2p-1)^r.
\]
\end{proposition}

\begin{proof}
Let $0\neq g\in G$.  If $g$ has order $p$, then $g\in pG=H$.  If $g$ has order
$p^2$, then $pg$ is a nonzero element of $\langle g\rangle\cap H$.  Hence every
nontrivial subgroup of $G$ meets $H$ nontrivially, and therefore $m(G,H)=1$.

The coordinate box
\[
  B_p^r=\{0,1,\ldots,p-1\}^r\subseteq(\mathbb Z/p^2\mathbb Z)^r
\]
is a transversal for $G/H$.  Its difference support is
\[
  D(B_p^r)=\{-(p-1),\ldots,p-1\}^r,
\]
which has size $(2p-1)^r$.  This proves the upper bound.
\end{proof}

Thus, for odd $p$, the pair
\[
  \bigl((\mathbb Z/p^2\mathbb Z)^2,
        p(\mathbb Z/p^2\mathbb Z)^2\bigr)
\]
is the first same-prime residual case not covered by the cyclic formula, the
two-torsion obstruction, or the one-nonsplit-coordinate product theorem.  The
rest of the paper rewrites its transversals as graphs
\(\mathbb F_p^2\to\mathbb F_p^2\) and studies the resulting carry-corrected
finite-field difference images.

%% file: section6-odd-square-plane.tex
\section{The odd square-plane problem}
\label{sec:odd-square-plane}

We now specialize to the first unresolved same-prime residual case.  Throughout
this section $p$ is an odd prime,
\[
  G_p=(\mathbb Z/p^2\mathbb Z)^2,
  \qquad
  H_p=pG_p.
\]
By Proposition~\ref{prop:residual-square-modulus-pairs}, the pair is residual
and the coordinate box gives
\[
  \delta(G_p,H_p)\leq (2p-1)^2.
\]
We rewrite its transversals in finite-field terms and prove a uniform lower
bound towards the box value. The resulting corrected-derivative images are finite-field image sets;
these are adjacent to the sum-product/expander literature
\cites{BourgainKatzTao2004,MurphyPetridis2017} and, for graph/function
language to relative difference sets and planar functions works such as
\cite{PottSchmidtZhou2014}.

\subsection{Carry-corrected derivative images}
\label{subsec:carry-corrected-derivatives}

We identify the quotient $G_p/H_p$ with $\mathbb F_p^2$.  For
$x\in\mathbb F_p^2$, let $[x]\in\{0,1,\ldots,p-1\}^2$ denote the coordinatewise
standard lift.  Every normalized transversal for $G_p/H_p$ has a unique form
\[
  T_f=\{[x]+p[f(x)]:x\in\mathbb F_p^2\},
\]
where $f:\mathbb F_p^2\to\mathbb F_p^2$ satisfies $f(0)=0$.  This normalization
is harmless for difference supports: translating $T_f$ by an element of $pG_p$
subtracts a constant from $f$.

For $u,x\in\mathbb F_p^2$, define the carry vector
\[
  c_u(x)=\frac{[x+u]-[x]-[u]}{p}\in\mathbb F_p^2.
\]
Here the numerator is computed in $\mathbb Z^2$ using standard representatives;
its coordinates are either $0$ or $-p$, so $c_u(x)$ has coordinates $0$ or $-1$
modulo $p$.  We then define the carry-corrected derivative
\[
  d_u f(x)=f(x+u)-f(x)+c_u(x),
\]
and its image
\[
  A_u(f)=\{d_u f(x):x\in\mathbb F_p^2\}\subseteq\mathbb F_p^2.
\]

The finite-field parametrisation gives an exact fibre decomposition of the
difference support.
\begin{lemma}
\label{lem:square-plane-fibre-decomposition}
For every $f:\mathbb F_p^2\to\mathbb F_p^2$ one has
\[
  |T_f-T_f|=\sum_{u\in\mathbb F_p^2}|A_u(f)|.
\]
Moreover $A_0(f)=\{0\}$.
\end{lemma}

\begin{proof}
The difference between the representative above $x+u$ and the representative
above $x$ is
\[
\begin{aligned}
  {[x+u]}+p[f(x+u)]-[x]-p[f(x)]
  &= [u]+p\bigl(f(x+u)-f(x)+c_u(x)\bigr) \\
  &= [u]+p\,d_u f(x)
\end{aligned}
\]
in $(\mathbb Z/p^2\mathbb Z)^2$.  In the quotient fibre
$u\in G_p/H_p\simeq\mathbb F_p^2$, the possible $p$-parts are exactly
$A_u(f)$.  The quotient fibres are disjoint, so summing over all $u$ gives the
formula.  For $u=0$ one has $c_0(x)=0$ and
$d_0f(x)=0$, hence $A_0(f)=\{0\}$.
\end{proof}

The coordinate box corresponds to $f=0$, and its support has size
$(2p-1)^2$.  This motivates the central conjecture.

\begin{conjecture}
\label{conj:odd-square-plane-box}
For every odd prime $p$,
\[
  \delta\bigl((\mathbb Z/p^2\mathbb Z)^2,
              p(\mathbb Z/p^2\mathbb Z)^2\bigr)=(2p-1)^2.
\]
Equivalently, for every function $f:\mathbb F_p^2\to\mathbb F_p^2$,
\[
  \sum_{u\in\mathbb F_p^2}|A_u(f)|\geq (2p-1)^2.
\]
\end{conjecture}

The general residual fibre lower bound gives only
\[
  |T_f-T_f|\geq 2p^2-1.
\]
Thus the conjecture asks for an additional
\[
  (2p-1)^2-(2p^2-1)=2(p-1)^2
\]
forced differences coming from the rank-two same-prime geometry.

\subsection{Cycle and curl identities}
\label{subsec:cycle-curl-identities}

The deterministic input is a pair of elementary identities using only that the
corrected derivatives come from a global quotient section.
\begin{lemma}
\label{lem:cycle-curl-two-point}
Let $p$ be an odd prime and let $f:\mathbb F_p^2\to\mathbb F_p^2$.  For all
$u,v,x\in\mathbb F_p^2$ one has the curl identity
\[
  d_u f(x+v)-d_u f(x)=d_v f(x+u)-d_v f(x).
\]
Moreover, for every nonzero $u\in\mathbb F_p^2$ and every affine $u$-line,
\[
  x+\langle u\rangle=\{x+iu:0\leq i\leq p-1\},
\]
one has the cycle identity
\[
  \sum_{i=0}^{p-1} d_u f(x+iu)=-u.
\]
Consequently:
\begin{enumerate}
  \item $|A_u(f)|\geq 2$ for every $u\neq 0$;
  \item if $|A_u(f)|=2$, then $A_u(f)$ is contained in an affine line parallel
  to $\langle u\rangle$.
\end{enumerate}
\end{lemma}

\begin{proof}
Expanding the definition of $d_u f$ gives
\[
\begin{aligned}
  d_u f(x+v)-d_u f(x)
  &=[f(x+u+v)-f(x+v)]-[f(x+u)-f(x)] \\
  &\quad +c_u(x+v)-c_u(x),
\end{aligned}
\]
and the analogous expression with $u$ and $v$ interchanged.  The $f$-terms are
identical.  The carry terms are also identical, since
\[
  c_u(x+v)-c_u(x)
  =\frac{[x+u+v]-[x+v]-[x+u]+[x]}{p}
  =c_v(x+u)-c_v(x).
\]
This proves the curl identity.

For the cycle identity, sum in $G_p$ the $p$ successive differences from the
representative above $x+iu$ to the representative above $x+(i+1)u$:
\[
  [u]+p\,d_u f(x+iu),
  \qquad 0\leq i\leq p-1.
\]
These differences go once around the affine cycle, so their sum is zero in
$G_p$:
\[
  p[u]+p\sum_{i=0}^{p-1}d_u f(x+iu)=0\pmod {p^2}.
\]
Dividing by $p$ modulo $p$ gives
\[
  u+\sum_{i=0}^{p-1}d_u f(x+iu)=0
\]
in $\mathbb F_p^2$.

If $A_u(f)=\{a\}$ for some $u\neq 0$, then the cycle identity gives
$pa=0=-u$, a contradiction.  Hence $|A_u(f)|\geq 2$ for every nonzero $u$.

Finally suppose $A_u(f)=\{a,b\}$.  On a fixed affine $u$-line, let $n$ be the
number of points at which $d_u f$ takes the value $a$.  Then the other $p-n$
points take the value $b$, and the cycle identity gives
\[
  na+(p-n)b=n(a-b)=-u.
\]
The cases $n=0$ and $n=p$ would make the left-hand side zero, so $n$ is nonzero
modulo $p$.  Thus
\[
  a-b=-n^{-1}u\in\langle u\rangle.
\]
Therefore $A_u(f)$ lies in an affine line parallel to $\langle u\rangle$.
\end{proof}

\subsection{A deterministic lower bound}
\label{subsec:deterministic-square-plane-bound}

Two independent quotient directions with minimal derivative images already force
the conjectural box lower bound.

\begin{lemma}
\label{lem:two-independent-two-point-directions}
Let $p$ be an odd prime and let $f:\mathbb F_p^2\to\mathbb F_p^2$.  Suppose
there are two linearly independent directions $u,v\in\mathbb F_p^2$ such that
\[
  |A_u(f)|=|A_v(f)|=2.
\]
Then
\[
  |T_f-T_f|\geq (2p-1)^2.
\]
\end{lemma}

\begin{proof}
By Lemma~\ref{lem:cycle-curl-two-point}, the two images $A_u(f)$ and $A_v(f)$
are respectively contained in affine lines parallel to $\langle u\rangle$ and
$\langle v\rangle$.  Therefore
\[
  d_u f(x+v)-d_u f(x)\in\langle u\rangle,
  \qquad
  d_v f(x+u)-d_v f(x)\in\langle v\rangle.
\]
By the curl identity these two quantities are equal.  Since $u$ and $v$ are
independent, $\langle u\rangle\cap\langle v\rangle=\{0\}$, so both quantities
vanish.  Hence $d_u f$ is invariant in the $v$-direction and $d_v f$ is
invariant in the $u$-direction.

Every nonzero $w\in\mathbb F_p^2$ can be written uniquely as
\[
  w=\lambda u+\mu v,
  \qquad \lambda,\mu\in\mathbb F_p.
\]
If exactly one of $\lambda,\mu$ is nonzero, then
Lemma~\ref{lem:cycle-curl-two-point} gives $|A_w(f)|\geq 2$.

Assume now that $\lambda\mu\neq 0$, and choose representatives
$\lambda,\mu\in\{1,\ldots,p-1\}$.  Following first $\lambda$ steps in the
$u$-direction and then $\mu$ steps in the $v$-direction, define
\[
  B_\lambda(x)=\sum_{i=0}^{\lambda-1}d_u f(x+iu),
  \qquad
  C_\mu(x)=\sum_{j=0}^{\mu-1}d_v f(x+\lambda u+jv).
\]
The direct corrected derivative $d_w f(x)$ differs from
$B_\lambda(x)+C_\mu(x)$ by a constant depending only on $\lambda,\mu,u,v$ and on
our standard-lift carry convention, not on $x$.  This is just the telescoping of
the \(f\)-increments along the chosen path; the discrepancy between the path
carry and the standard carry for \(w\) is independent of the starting point.
Hence $A_w(f)$ contains a translate of the sum of the two images of $B_\lambda$
and $C_\mu$.
Since $d_u f$ is invariant in the $v$-direction, $B_\lambda$ depends only on the
$u$-coordinate of $x$ and its image lies in an affine line parallel to
$\langle u\rangle$.  Moreover $B_\lambda$ differs from the direct derivative
$d_{\lambda u}f$ by a constant, so
\[
  |\operatorname{im} B_\lambda|=|A_{\lambda u}(f)|\geq 2.
\]
Similarly, because $d_v f$ is invariant in the $u$-direction, $C_\mu$ depends
only on the $v$-coordinate, has image contained in an affine line parallel to
$\langle v\rangle$, and
\[
  |\operatorname{im} C_\mu|=|A_{\mu v}(f)|\geq 2.
\]
As the \(u\)- and \(v\)-coordinates of \(x\) vary independently, the pair
\((B_\lambda(x),C_\mu(x))\) runs over
\(\operatorname{im}B_\lambda\times\operatorname{im}C_\mu\).
Since the two image directions are parallel to the independent lines
$\langle u\rangle$ and $\langle v\rangle$, addition is injective on their
product.  Therefore
\[
  |A_w(f)|\geq
  |\operatorname{im} B_\lambda|\,|\operatorname{im} C_\mu|
  \geq 4
\]
for every mixed direction $w=\lambda u+\mu v$ with $\lambda\mu\neq 0$.

There are $p-1$ nonzero multiples of $u$, $p-1$ nonzero multiples of $v$, and
$(p-1)^2$ mixed directions.  Using Lemma~\ref{lem:square-plane-fibre-decomposition},
we get
\[
\begin{aligned}
  |T_f-T_f|
  &=1+\sum_{w\neq 0}|A_w(f)| \\
  &\geq 1+2(p-1)+2(p-1)+4(p-1)^2 \\
  &=(2p-1)^2.
\end{aligned}
\]
\end{proof}

This dichotomy gives an unconditional lower bound for every odd prime.
\begin{theorem}
\label{thm:unconditional-square-plane-bound}
For every odd prime $p$,
\[
  \delta\bigl((\mathbb Z/p^2\mathbb Z)^2,
              p(\mathbb Z/p^2\mathbb Z)^2\bigr)
  \geq 3p^2-p-1.
\]
Equivalently, every transversal $T$ for $G_p/H_p$ satisfies
\[
  |D(T)|\geq 3p^2-p-1.
\]
\end{theorem}

\begin{proof}
Normalize the transversal and write it as $T=T_f$.  Let
\[
  S=\{u\in\mathbb F_p^2\setminus\{0\}: |A_u(f)|=2\}.
\]
If $S$ contains two independent directions, then
Lemma~\ref{lem:two-independent-two-point-directions} gives the stronger bound
\[
  |T_f-T_f|\geq (2p-1)^2.
\]
Since
\[
  (2p-1)^2-(3p^2-p-1)=(p-1)(p-2)\geq 0
\]
for $p\geq 3$, the desired estimate follows in this case.

It remains to consider the case where $S$ contains no two independent
directions.  Then $S$ is contained in a single one-dimensional subspace of
$\mathbb F_p^2$, so
\[
  |S|\leq p-1.
\]
By Lemma~\ref{lem:cycle-curl-two-point}, every nonzero direction contributes at
least $2$, and every nonzero direction outside $S$ contributes at least $3$.
Together with $A_0(f)=\{0\}$, this gives
\[
\begin{aligned}
  |T_f-T_f|
  &=1+\sum_{u\neq 0}|A_u(f)| \\
  &\geq 1+2|S|+3\bigl(p^2-1-|S|\bigr) \\
  &=3p^2-2-|S| \\
  &\geq 3p^2-p-1.
\end{aligned}
\]
This proves the theorem.
\end{proof}

The remaining gap to the conjectural box value is \((p-1)(p-2)\);
Section~\ref{sec:asymptotic-evidence} gives asymptotic evidence for
Conjecture~\ref{conj:odd-square-plane-box}.

\subsection{Small-prime square-plane certificates}
\label{subsec:small-prime-certificates}

We verified the conjecture for \(p=3\) and \(p=5\) by deterministic Python
computations using the graph parametrisation and fibre decomposition above.
For \(p=5\) the search also uses the cycle and curl constraints from this
section and the action of \(\operatorname{GL}_2(\mathbb F_5)\).  Since \(T-T\)
is negation-invariant and \(0\) is the only self-inverse element in an odd-order
group, \(|T-T|\) is odd.  Thus it suffices to exclude the largest odd support
below the box value: the computations rule out \(|T_f-T_f|\le 23\) for \(p=3\)
and \(|T_f-T_f|\le 79\) for \(p=5\).  The coordinate boxes attain \(25\) and
\(81\), so
\[
\delta\bigl((\mathbb Z/9\mathbb Z)^2,3(\mathbb Z/9\mathbb Z)^2\bigr)=25,
\qquad
\delta\bigl((\mathbb Z/25\mathbb Z)^2,5(\mathbb Z/25\mathbb Z)^2\bigr)=81.
\]

 The Python scripts used and the results could be consulted on our online companion \url{github.com/georgeturcasubb/transversals}.

We could not complete a verification of the conjecture for $p=7$ due to the high computational complexity. However, in these case we conducted many unsuccessful searches for transversals that could give rise to counter-examples.

%% file: section7-asymptotic-evidence.tex
\section{Asymptotic evidence for the square-plane conjecture}
\label{sec:asymptotic-evidence}

The exact values in Subsection~\ref{subsec:small-prime-certificates} cover the
first two odd primes.  We add two asymptotic pieces of evidence: a uniformly
random lifting satisfies the box lower bound with probability tending to one,
and no fixed integer-polynomial rule produces counterexamples for all
sufficiently large primes.

\subsection{Random liftings}
\label{subsec:random-liftings}

Call a nonzero direction $u=(u_1,u_2)\in\mathbb F_p^2$ \emph{mixed} if
$u_1u_2\neq 0$.  For a function $f:\mathbb F_p^2\to\mathbb F_p^2$, define
\[
  B(f)=\{u\in\mathbb F_p^2:u_1u_2\neq 0\text{ and } |A_u(f)|\leq 3\}.
\]
Thus $B(f)$ records mixed directions whose corrected derivative image is too
small for the box contribution.

\begin{theorem}
\label{thm:random-square-plane-box-bound}
Let $p$ be an odd prime.  Choose $f:\mathbb F_p^2\to\mathbb F_p^2$ uniformly at
random, with the values $f(x)$ independent and uniformly distributed.  Then
\[
  \Pr_f\bigl(|T_f-T_f|<(2p-1)^2\bigr)
  \leq
  p^8\left(\frac{p^2 3^p}{p^{2p}}\right)^p.
\]
In particular,
\[
  \Pr_f\bigl(|T_f-T_f|\geq (2p-1)^2\bigr)\longrightarrow 1
  \qquad (p\to\infty).
\]
The same conclusion holds for uniformly random normalized liftings $f(0)=0$.
\end{theorem}

\begin{proof}
Before normalization, uniform functions give the uniform model on transversals
for $G_p/H_p$; subtracting $f(0)$ only translates $T_f$ by an element of \(pG_p\).

Fix $f$.  By Lemma~\ref{lem:square-plane-fibre-decomposition},
\[
  |T_f-T_f|=\sum_{u\in\mathbb F_p^2}|A_u(f)|.
\]
The zero direction contributes $1$, and
Lemma~\ref{lem:cycle-curl-two-point} gives at least $2$ from every nonzero
direction.  There are $2(p-1)$ nonzero axis directions and $(p-1)^2$ mixed
directions; the mixed directions outside $B(f)$ contribute at least $4$, while
those in $B(f)$ still contribute at least $2$.  Hence
\[
\begin{aligned}
  |T_f-T_f|
  &\geq 1+2\cdot 2(p-1)+2|B(f)|+4\bigl((p-1)^2-|B(f)|\bigr) \\
  &= (2p-1)^2-2|B(f)|.
\end{aligned}
\]
Thus \(B(f)=\varnothing\) is sufficient for the box lower bound, and failure
forces \(B(f)\neq\varnothing\).

We bound this last event.  Fix a mixed direction $u$ and a set
$S\subseteq\mathbb F_p^2$ with $|S|\leq 3$.  The translation \(x\mapsto x+u\)
partitions \(\mathbb F_p^2\) into \(p\) disjoint \(p\)-cycles.  On one cycle,
write \(x_i=x_0+iu\) and \(y_i=f(x_i)\).  The condition \(d_u f(x_i)\in S\) is
\[
  y_{i+1}-y_i\in S-c_u(x_i).
\]
After choosing \(y_0\), there are at most \(3\) choices for each edge increment;
ignoring the closing condition only overcounts.  Thus one cycle has probability
at most
\[
  \frac{p^2 3^p}{p^{2p}}.
\]
The $p$ cycles use disjoint values of $f$, hence are independent.  The
probability that this fixed $S$ contains all values of $A_u(f)$ is therefore at
most
$
  \left(\frac{p^2 3^p}{p^{2p}}\right)^p$.
Hence there are at most
$
  \sum_{j=0}^3 \binom{p^2}{j}\leq p^6
$
subsets of $\mathbb F_p^2$ of size at most $3$. It follows that, for the fixed mixed
$u$,
\[
  \Pr_f\bigl(|A_u(f)|\leq 3\bigr)
  \leq
  p^6\left(\frac{p^2 3^p}{p^{2p}}\right)^p.
\]
Finally, there are $(p-1)^2<p^2$ mixed directions.  A union bound gives
\[
  \Pr_f(B(f)\neq\varnothing)
  \leq
  p^8\left(\frac{p^2 3^p}{p^{2p}}\right)^p.
\]
Since failure of the box lower bound implies $B(f)\neq\varnothing$, this proves
the estimate.  Its logarithm is
\(-2p^2\log p+p^2\log 3+O(p\log p)\), which tends to \(-\infty\).
\end{proof}

\subsection{Fixed-polynomial liftings}
\label{subsec:fixed-polynomial-liftings}

The random theorem shows that counterexamples, if they exist, are rare among all
functions.  We next rule out another natural source: liftings defined by one
fixed polynomial rule as the prime varies.

Let $g=(g_1,g_2)\in\mathbb Z[X,Y]^2$ be fixed.  For a prime $p$, let
$g_p:\mathbb F_p^2\to\mathbb F_p^2$ be its reduction modulo $p$, and set
\[
  T_{g,p}=\{[z]+p[g_p(z)]:z\in\mathbb F_p^2\}
  \subseteq (\mathbb Z/p^2\mathbb Z)^2.
\]
The word fixed is essential: for a single prime \(p\), every function
\(\mathbb F_p^2\to\mathbb F_p^2\) has polynomial representatives of degree less
than \(p\) in each variable, for instance by finite-field interpolation.  The theorem excludes only one
integer-polynomial rule independent of \(p\). This fixed-polynomial setting is close to value-set questions for
polynomial maps over finite fields \cite{MullenWanWang2013} and to finite-field polynomial/rational
image-expansion results
\cite{BukhTsimerman2012}.

\begin{theorem}
\label{thm:fixed-polynomial-square-plane}
For every fixed $g=(g_1,g_2)\in\mathbb Z[X,Y]^2$, there is an integer $B_g$
such that, for every prime $p>B_g$,
\[
  |T_{g,p}-T_{g,p}|\geq (2p-1)^2.
\]
\end{theorem}

The proof uses a nonzero-Jacobian image bound, large-rectangle estimates, and a
degenerate-Jacobian classification.
\begin{lemma}
\label{lem:image-bound-nonzero-jacobian}
Fix $E\geq 1$.  Let
\[
  F=(P,Q):\mathbb A^2_{\mathbb F_p}\to\mathbb A^2_{\mathbb F_p}
\]
be a polynomial map with $\deg P,\deg Q\leq E$ and $\det JF\not\equiv 0$.  If
$R\subseteq\mathbb F_p^2$, then
\[
  |F(R)|\geq \frac{|R|-C_Ep}{E^2},
\]
where one may take $C_E=\max\{1,2(E-1)\}$.
\end{lemma}

\begin{proof}
Let $\Delta=\det JF$.  Its zero set has at most $C_Ep$ points by the elementary
polynomial zero bound over finite fields.
For $a\in\mathbb F_p^2$, the fibre $F^{-1}(a)$ is cut out by two plane curves of
degrees at most $E$.  Discard any common fibre components contained in
\(\{\Delta=0\}\); those points are already excluded below.  Off the critical
curve, the remaining fibre curves have no common component: otherwise their
gradients would be dependent at a smooth point of that component, forcing
$\Delta$ to vanish generically there.  Bezout's theorem for plane curves
\cite{CoxLittleOShea2025IVA} then bounds every noncritical fibre by $E^2$
geometric, hence rational, points.
Thus the points of $R$ outside $\{\Delta=0\}$ map with fibres of size at most
$E^2$, and
\[
  |F(R)|\geq \frac{|R|-C_Ep}{E^2}.
\]
\end{proof}

We also use three elementary estimates on large rectangles.  Let
\(I,J\subseteq\mathbb F_p\) be intervals of consecutive standard residues, each
of size at least \(p/2\), and put \(R=I\times J\).
\begin{enumerate}
  \item A nonzero linear form \(\ell\) has \(|\ell(R)|\geq p/2\).  This is
  immediate if one coefficient vanishes, and otherwise follows from
  Cauchy--Davenport \cite{TaoVu2006}.
  \item If \(P\in\mathbb F_p[S]\) is nonconstant of degree at most \(d\), then
  \(|P(S_0)|\geq |S_0|/d\) for every \(S_0\subseteq\mathbb F_p\).
  \item If \(Q\in\mathbb F_p[X,Y]\) is nonconstant of degree at most \(d\), then
  every nonempty fibre has at most \(dp\) points by the elementary zero bound, so
  \(|Q(R)|\geq |R|/(dp)\geq p/(4d)\).
\end{enumerate}

The second input handles the degenerate Jacobian case, using singular subspaces
of \(2\times 2\) matrices.
\begin{lemma}
\label{lem:singular-subspaces-matrices}
Let $k$ be an infinite field and let $L\subseteq M_2(k)$ be a linear subspace
such that every element of $L$ is singular.  Then either all matrices in $L$ have
a common nonzero kernel vector, or all their images lie in a common line in
$k^2$.
\end{lemma}

\begin{proof}
If $L=0$, both alternatives hold.  Otherwise choose a nonzero rank-one matrix
in $L$ and change source and target bases so that it is $
  E=\begin{pmatrix}1&0\\0&0\end{pmatrix}$.
For $
  B=\begin{pmatrix}a&b\\ c&d\end{pmatrix}\in L$,
the matrix $E+tB$ is singular for every $t\in k$.  Since $k$ is infinite,
\[
  \det(E+tB)=td+t^2(ad-bc)
\]
vanishes as a polynomial in $t$.  Hence $d=0$.  Since $B$ is singular, also
$bc=0$.  Linearity forces either $b=0$ for all $B\in L$, or $c=0$ for all
$B\in L$; otherwise a linear combination would have both off-diagonal entries
nonzero and would be nonsingular.  These are exactly the common-kernel and
common-image cases.
\end{proof}

This gives the needed classification for polynomial maps with singular
Jacobian differences.
\begin{lemma}
\label{lem:degenerate-jacobian-classification}
Let $k$ be a field of characteristic zero, and let $g:k^2\to k^2$ be a
polynomial map satisfying
\[
  \det(Jg(w)-Jg(z))=0
  \qquad \text{for all } z,w\in k^2.
\]
Then there are $A\in M_2(k)$ and $c\in k^2$ such that one of the following holds:
\[
  g(z)=Az+G(\ell(z))+c
\]
for a nonzero linear form $\ell:k^2\to k$ and a one-variable polynomial map
$G:k\to k^2$, or
\[
  g(z)=Az+w\varphi(z)+c
\]
for some nonzero $w\in k^2$ and some polynomial $\varphi\in k[X,Y]$.
\end{lemma}

\begin{proof}
Fix $z_0$ and let $L$ be the linear span of
\[
  \{Jg(z)-Jg(z_0):z\in k^2\}.
\]
The determinant quadratic form vanishes on each generator and on each
difference of two generators.  Its polar form therefore vanishes on pairs of
generators, so the determinant vanishes on all of $L$.
Lemma~\ref{lem:singular-subspaces-matrices} gives a common kernel vector or a
common image line.

Write $Jg(z)=A+N(z)$ with $N(z)\in L$.  If $L$ has a common kernel vector
$r\neq 0$, the directional derivative of $g-Az$ along $r$ is zero.  After a
linear change of coordinates, $g-Az$ is independent of one coordinate and has
the form $G(\ell(z))+c$.  If $L$ has common image line $kw$, all derivatives of
$g-Az$ lie in $kw$.  A linear form annihilating $w$ is constant on $g-Az$, so
$g-Az-c$ takes values in $kw$ and $g(z)=Az+w\varphi(z)+c$.
\end{proof}

\begin{proof}[Proof of Theorem~\ref{thm:fixed-polynomial-square-plane}]
We first remove affine terms, which do not change the support size.  Adding a
constant translates \(T_{g,p}\) by an element of \(pG_p\).  Adding a linear term
\(Az\) applies the automorphism
\[
  U_A(t)=t+pA\overline t,
  \qquad t\in(\mathbb Z/p^2\mathbb Z)^2,
\]
where \(\overline t\) denotes reduction modulo \(p\); its inverse is \(U_{-A}\).
Thus both operations preserve \(|T_{g,p}-T_{g,p}|\), and affine-linear liftings
have the coordinate-box support \((2p-1)^2\).  Let \(D=\deg g\).  If \(D\leq 1\),
we are done.

Assume first that $D\geq 2$ and
\[
  \Delta_g(z,v)=\det\bigl(Jg(z+v)-Jg(z)\bigr)
\]
is not the zero polynomial over $\mathbb Q$.  Expand $\Delta_g$ in the
$z$-variables and choose a nonzero coefficient polynomial $H(v)$.  After
excluding finitely many primes depending only on $g$, the reduction of $H$ is
nonzero; hence, for all sufficiently large $p$, all but $O_g(p)$ directions
$v\in\mathbb F_p^2$ have $H(v)\neq 0$.  For these directions, the finite
difference
\[
  F_v(z)=g_p(z+v)-g_p(z)
\]
has degree at most $D-1$ and nonzero Jacobian determinant.

For each good $v$, choose a carry cell \(R_v\): a rectangular cell of standard
residue representatives on which \(c_v(z)\) is constant, with both side lengths
at least \(p/2\), so \(|R_v|\geq p^2/4\).  On \(R_v\), the map
\(z\mapsto d_v g_p(z)\) is a constant translate of \(F_v(z)\).  By
Lemma~\ref{lem:image-bound-nonzero-jacobian}, with $E=D-1$,
\[
  |A_v(g_p)|\geq \frac{p^2/4-C_Ep}{E^2}.
\]
There are \(p^2-O_g(p)\) good directions, each contributing \(\gg_g p^2\), so
Lemma~\ref{lem:square-plane-fibre-decomposition} gives
\[
  |T_{g,p}-T_{g,p}|=\sum_v |A_v(g_p)|\gg_g p^4.
\]
For all sufficiently large primes this exceeds \((2p-1)^2\), proving the
generic case.

It remains to treat the case $\Delta_g\equiv 0$.  Then
\[
  \det(Jg(w)-Jg(z))=0
\]
for all $z,w$ over $\mathbb Q$, so
Lemma~\ref{lem:degenerate-jacobian-classification} applies.  Excluding the
finitely many primes where the rational classification data have denominator
problems, where the nonzero linear form or target vector vanishes after
reduction, where the relevant leading forms vanish, or where the characteristic
is too small for the degree bounds, we remove the affine part as above.  One of
two forms remains.

First suppose \(g(z)=G(\ell(z))\), with \(\ell\neq 0\).  If
\(m=\deg G\leq 1\), the map is affine-linear, already handled.  Otherwise, for
all sufficiently large primes and every \(v\) with \(\ell(v)\neq 0\), the
difference \(G(S+\ell(v))-G(S)\) has a nonconstant coordinate of degree
\(m-1\).  On a large carry cell, the linear-form and univariate image
estimates above give
\[
  |A_v(g_p)|\geq \frac{p}{2(m-1)}.
\]
The \(p^2-p\) such directions contribute \(\gg_g p^3\), which exceeds
\((2p-1)^2\) for all sufficiently large \(p\).

Second suppose \(g(z)=w\varphi(z)\), with \(w\neq 0\).  If
\(m=\deg\varphi\leq 1\), the map is affine-linear, already handled.  Otherwise,
for all sufficiently large primes the highest homogeneous part \(\varphi_m\)
remains nonconstant.  The directions \(v\) for which \(D_v\varphi_m\) vanishes
identically form a proper linear subspace, so all but at most \(p\) directions
give a nonconstant difference \(\varphi(z+v)-\varphi(z)\) of degree at most
\(m-1\).  On a large carry cell, \(d_v g_p\) is a constant translate of
\(w(\varphi(z+v)-\varphi(z))\), so multiplication by \(w\neq 0\) preserves
cardinality and the bivariate image estimate gives
\[
  |A_v(g_p)|\geq \frac{p}{4(m-1)}.
\]
Again the total contribution is \(\gg_g p^3\), which exceeds \((2p-1)^2\) for
all sufficiently large \(p\).

The generic and degenerate cases give an integer $B_g$ such that the claimed
inequality holds for every prime $p>B_g$.
\end{proof}

\begin{remark}
These results do not settle Conjecture~\ref{conj:odd-square-plane-box}: they
exclude random counterexamples with high probability and fixed algebraic rules
for large \(p\), while the conjecture allows arbitrary \(p\)-dependent liftings.
\end{remark}

%% file: section8-open-problems-and-appendices.tex
\section{Open problems and outlook}
\label{sec:open-problems-outlook}

The cyclic and one-nonsplit-coordinate cases are settled by the fibre lower
bound, singleton reduction, and Kneser estimate.  The remaining obstruction is
the same-prime nonsplit rank-two case: the odd-prime box statement is still
conjectural, and the fixed-prime certificates and asymptotic evidence have only
the limited scope proved above. The fixed-polynomial evidence also resonates with the broader finite-field
polynomial-method tradition, exemplified by Dvir's finite-field Kakeya
breakthrough \cite{Dvir2009}; however, our result is not a
Kakeya theorem and uses only the polynomial image constraints developed above.

\subsection{Square-plane problem}

The central open problem is Conjecture~\ref{conj:odd-square-plane-box}.  In the
notation of Lemma~\ref{lem:square-plane-fibre-decomposition}, it asks whether
\(\sum_u |A_u(f)|\geq (2p-1)^2\) for every \(f\), with equality for the
coordinate box.  Theorem~\ref{thm:unconditional-square-plane-bound} leaves gap
\((p-1)(p-2)\), so any counterexample would have to realize a structured loss
among low corrected-derivative images.

Sharper bounds for low corrected-derivative images may require incidence
technology beyond the cycle-curl identities used here, such as Rudnev's \cite{Rudnev2018}
point-plane incidence theorem and the Stevens--de Zeeuw point-line
incidence bound over arbitrary fields
\cite{StevensDeZeeuw2017}.

\begin{problem}
\label{prob:low-direction-rigidity}
Let $p$ be odd and let $f:\mathbb F_p^2\to\mathbb F_p^2$.  Classify, or sharply
bound, the oriented nonzero directions \(u\) for which \(|A_u(f)|\leq 3\).  In
particular, decide whether any loss from these low images is necessarily
compensated by larger corrected-derivative images in the remaining directions.
\end{problem}

\subsection{Higher-rank square-modulus quotients}

The coordinate box suggests the following higher-rank analogue.

\begin{conjecture}
\label{conj:higher-rank-box}
For every odd prime $p$ and every $r\geq 1$,
\[
  \delta\bigl((\mathbb Z/p^2\mathbb Z)^r,
  p(\mathbb Z/p^2\mathbb Z)^r\bigr)=(2p-1)^r.
\]
\end{conjecture}

For higher-rank quotients, one natural next layer could be iterated
sumset-growth technology in the Plünnecke--Ruzsa tradition
\cites{Plunnecke1970,Ruzsa1989}, together with arbitrary-abelian-group
small-doubling structure \cite{GreenRuzsa2007}.

The cases \(r=1\) and \(r=2\) are respectively the cyclic residual theorem and
Conjecture~\ref{conj:odd-square-plane-box}.  For \(r\geq 3\), the same
graph-and-fibre identity holds, and the rank-two argument suggests looking for a
frame of low directions and product-type lower bounds for mixed directions.

\subsection{Primitive quotients and residual geometries}

The primitive quotient reduction separates singleton directions from the
primitive quotient part, so the remaining classification can be phrased in
primitive terms.

\begin{problem}
\label{prob:primitive-support}
Classify the pairs $H\leq G$ for which every primitive transversal $T$ satisfies
\(|D(T)|\geq 2|G/H|-1\), and identify the primitive quotient geometries that
force a strict improvement.
\end{problem}

Cyclic residual quotients attain the bound \(2|G/H|-1\); residual quotients
containing a \(C_2\times C_2\) plane do not.  The product results control one
nonsplit cyclic coordinate with arbitrary split factors, leaving the next case.

\begin{problem}
\label{prob:multiple-nonsplit-coordinates}
Let $p$ be prime, \(G=\prod_{i=1}^r C_{p^{a_i}}\), and
\(H=\prod_{i=1}^r p^{b_i}C_{p^{a_i}}\), with \(1\leq b_i<a_i\) for at least two
indices.  Determine \(\delta(G,H)\), or give sharp lower bounds in terms of the
nonsplit \(p\)-primary coordinates.
\end{problem}

The square-plane case is the smallest instance of
Problem~\ref{prob:multiple-nonsplit-coordinates}.  Resolving it would clarify
the first residual nonsplit geometry left open by the transversal difference
number.